\documentclass[12pt]{amsart}
\usepackage[utf8]{inputenc}

\makeatletter
\def\@settitle{\begin{center}%
		\baselineskip14\p@\relax
		\normalfont\Large
		\@title
	\end{center}%
}
\makeatletter
\def\@settitle{\begin{center}%
		\baselineskip14\p@\relax
		\normalfont\Large
		\@title
	\end{center}%
}
\makeatother

\usepackage{graphicx,epsfig}

\usepackage{bbm}
\usepackage[title]{appendix}
\usepackage{esvect}
\usepackage{tikz-cd}
\usepackage{fancyhdr}
\usepackage{amsfonts}
\usepackage{amsmath}
\usepackage{amssymb}
\usepackage{mathabx}
\usepackage{amsthm}
\usepackage{tikz-cd}
\usepackage{parskip}
\usepackage{times} 

\makeatletter
\newcommand\xleftrightarrow[2][]{%
  \ext@arrow 9999{\longleftrightarrowfill@}{#1}{#2}}
\newcommand\longleftrightarrowfill@{%
  \arrowfill@\leftarrow\relbar\rightarrow}
\makeatother

\usepackage{mathrsfs}
\usepackage{scalerel,amssymb}

\usepackage{amsmath,amssymb,latexsym, amsfonts, amscd, amsthm, xy}
\input{xy}
\xyoption{all}

\input{xy}
\xyoption{all}

\makeindex 

=650pt \headheight = 8pt \headsep = 10pt

\usepackage{fancyhdr}
\usepackage{amsfonts}
\usepackage{amsmath}
\usepackage{amssymb}

\usepackage{amsthm}
\usepackage{tikz}
\usepackage{tikz-cd}
\usepackage{parskip}
\usepackage{times} 

\usepackage[
bibstyle=alphabetic,
citestyle=alphabetic,
hyperref=true,
backref=false]{biblatex}
\usepackage{hyperref}
\hypersetup{pdftoolbar=true, pdftitle={GKform}, pdffitwindow=true, colorlinks=true, citecolor=green, filecolor=black, linkcolor=blue, urlcolor=blue, hypertexnames=false}
\usepackage{cleveref}
\makeindex 

\newcommand{\cal}{\mathcal}

\renewcommand{\rm}{\mathrm}

\newcommand{\Z}{\mathbb{Z}}
\newcommand{\Q}{\mathbb{Q}}
\newcommand{\R}{\mathbb{R}}
\newcommand{\C}{\mathbb{C}}

\newcommand{\G}{\mathbb{G}}

\newcommand{\GL}{\mathrm{GL}}
\newcommand{\St}{\mathrm{St}}
\newcommand{\OS}{\mathrm{OS}}

\DeclareMathSymbol{\mathinvertedexclamationmark}{\mathclose}{operators}{'074}
\DeclareMathSymbol{\mathexclamationmark}{\mathclose}{operators}{'041}

\makeatletter
\newcommand{\raisedmathinvertedexclamationmark}{%
  \mathclose{\mathpalette\raised@mathinvertedexclamationmark\relax}%
}
\newcommand{\raised@mathinvertedexclamationmark}[2]{%
  \raisebox{\depth}{$\m@th#1\mathinvertedexclamationmark$}%
}
\begingroup\lccode`~=`! \lowercase{\endgroup
  \def~}{\@ifnextchar`{\raisedmathinvertedexclamationmark\@gobble}{\mathexclamationmark}}
\mathcode`!="8000
\makeatother

=650pt \headheight = 8pt \headsep = 10pt

\newtheorem{thm}{Theorem}[section]
\newtheorem{prop}[thm]{Proposition}
\newtheorem{lem}[thm]{Lemma}

\theoremstyle{definition}
\newtheorem{rem}[thm]{Remark}

\numberwithin{equation}{section}
\title[A note on polyhedral cones and toric polylogarithms]{A note on polyhedral cones and toric polylogarithms}
\author{Peter Xu}

\usepackage{setspace}
\begin{document}

        \makeatletter
\@setabstract
\makeatother
\maketitle

    \begin{abstract}
        We extend some methods of our previous work on special elements in Milnor $K$-theory of algebraic tori, exhibiting in particular a $\GL_n(\Q)$-equivariant isomorphism between a chain complex of simplicial cones, computing the homology of $S^{n-1}$, and the trace-fixed part of the weight-$n$ Gersten complex for the Milnor $K$-theory of $\G_m^n$ over $\Q$. Via a relationship between graded pieces of algebras of cones and Steinberg modules, this refines a result of Charlton--Radchenko--Rudenko. 
    \end{abstract}

    \tableofcontents

    In Section 4 of \cite{CR}, Charlton--Radchenko--Rudenko give an injection from the Steinberg module $\St_n(\Q)$ for $\GL_n(\Q)$ into the ``top depth graded part'' of the rational Milnor $K$-theory of an algebraic $n$-torus $\G_m^n$ over $\Q$, as motivation for their treatment using a Steinberg symbol formalism for toric polylogarithms. In this short note, we show (Theorem \ref{thm:main}) that the methods of our previous article \cite{X2}, together with some classical results in algebraic geometry, suffice to prove that a certain ``oriented cover'' $C(n)$ of $\St_n(\Q)$ (coming from conical geometry) maps \emph{isomorphically} into the \emph{trace-fixed part} integral Milnor $K$-theory of $\G_m^n$ over $\Q$, and that these maps assemble into a morphism of complexes from a homology complex of an $(n-1)$-sphere to the motivic weight-$n$ Gersten complex of $\G_m^n$.

    This strongly suggests to us, as hinted in the introduction and appendices of \cite{CR}, that there is a close relation between $C(n)$ and the \emph{toric Aomoto polylogarithms} defined in loc. cit. Morally, we expect that integrating the de Rham regulators of our Milnor $K$-theory elements over linear simplices inside $\G_m^n$ should relate the latter polylogarithms with the Koszul dual of the homology of $S^{n-1}$, but making this precise outside of the case of general position and proving the necessary relations is the hard part; we are interested in exploring this in further work.
        
    We expect the same method proves analogous results parameterizing special trace-fixed motivic elements in Gersten complexes for universal Drinfeld $A$-modules (cf. \cite{SX}), as well as in the setting of elliptic curves (cf. \cite{X3}, \cite{X4}), though with Steinberg-like modules rather than spherical chains (and involving a somewhat more technical stabilization in the latter case).  
    
    \section{Filtrations on algebras of polyhedral cones}
    We first explain aspects of algebras of cones, extending considerations in \cite[Appendix A]{CR}. Parts of the following account we already explained in \cite[\S4]{X2}, and are standard material in convexity theory \cite{Bv}, with the exception of the pair of filtrations we construct. This section is mostly motivational, and explains why our main theorem refines that of \cite{CR}, though the precise relation between Steinberg modules and cone algebras may also be of independent interest.
    
    Let $V=\Q^n$, and let $\mathcal{C}(V)$ be the algebra of 
    functions on $V_\R$ generated by indicator functions $[C]$ of closed rational polyhedral cones $C\subset \R^n$, under pointwise multiplication and addition; these may be naturally decomposed into closed rational simplical cones. There is a natural ascending filtration by dimension of support 
    \[
    \Z = F_0\cal{C}(V) \subset F_1 \mathcal{C}(V) \subset F_1\mathcal{C}(V) \subset \ldots \subset F_{n} \mathcal{C}(V)
    \]
    for which $F_p\mathcal{C}(V)\cdot F_q\mathcal{C}(V) \subset F_{\min\{p,q\}} \mathcal{C}(V)$
    
    Following \cite[Chapter IV, \S1.5]{Bv},\footnote{Our source uses \emph{polarity} of cones instead of duality, but these differ only by the automorphism $-1$ of the ambient space $\R^n$.} there is also a second, different algebra structure $*$ on $\mathcal{C}(V)$ coming from \emph{convolution} of functions, such that 
    \[
        [C_1] * [C_2] = [C_1+C_2]
    \]
    for closed polyhedral cones $C_1$, $C_2$.
    Moreover, there is a duality involution
    \[
        D:\mathcal{C}(V) \to \cal{C}(V)
    \]
    defined on conical generators by $[C] \mapsto [C^\vee]$, where $C^\vee$ is the \emph{dual} cone consisting of all vectors pairing nonnegatively with all vectors in $C$. This duality involution exchanges the two algebra structures, in that $D(fg) = D(f)*D(g)$ for any $f,g\in \mathcal{C}(V)$. Under the action of $\GL_n(\Q)$ on $V$, we have $D\circ \gamma = (\gamma^{-1})^T \circ D$. 

    From the preceding descriptions, one sees that convolution product satisfies $F_p\cal{C}(V) * F_q\cal{C}(V) \subset F_{p+q}\cal{C}(V)$. We deduce also a second filtration 
    \[
    \Z = F^D_0 \mathcal{C}(V) \subset F^D_1 \mathcal{C}(V) \subset \ldots \subset F_n^D \cal{C}(V)
    \]
    by applying $D$ to the filtration $F_\bullet$. The conical dual of the set of dimension-at-most-$i$ closed cones is precisely the set of top-dimensional closed cones containing at least a dimension-$(n-i)$ subspace of $\R^n$ - these are what we called ``wedges'' in \cite{X2} -  so we may describe $F^D_i\cal{C}(V)$ as generated by such $(n-i)$-wedges. We formally deduce $F_p^D\cal{C}(V) \cdot  F_q^D\cal{C}(V) \subset F_{p+q}^D\cal{C}(V)$ and $F_p^D\cal{C}(V) *  F_q^D\cal{C}(V) \subset F_{\min\{p,q\}}^D\cal{C}(V)$ from duality.

    We therefore see that under pointwise multiplication, each $F_p$ is an ideal of $\mathcal{C}(V)$, and under convolution, each $F_p^D$ is an ideal of $\mathcal{C}(V)$, but neither has an associated graded algebra.
    
    We \emph{do} have an associated graded algebra under convolution $\rm{gr}\, \cal{C}(V)$ for the filtration $F$, and an associated graded algebra under pointwise multiplication $\rm{gr}^D\, \cal{C}(V)$, which are isomorphic via
    \[
        D: \rm{gr}\, \cal{C}(V) \to \rm{gr}^D\, \cal{C}(V).
    \]
    These graded algebras have systems of graded ideals $F^D_p$, respectively $F_p$, from descending these filtrations to the associated gradeds, which are identifed by $D$. We have a $\GL_n(\Q)$-equivariant isomorphism 
    \begin{equation} \label{eq:rel}
    \mathrm{gr}_n\, \mathcal{C}(V)/ F_{n-1}^D \xrightarrow{\sim} \St_n(\Q),
    \end{equation}
    given by sending
    \[
     [C(v_1,\ldots, v_n)]\mapsto \mathrm{sgn}\,\det(v_1,\ldots, v_n)[v_1,\ldots, v_n]
    \]
    as described in \cite[Proposition 91]{CR}. By duality, one also obtains an isomorphism
    \[
    \mathrm{gr}^D_n\, \mathcal{C}(V)/F_{n-1} \xrightarrow{\sim} \St_n(\Q)
    \]
    which is $\GL_n(\Q)$-equivariant only after twisting by the inverse transpose on one side. Composing these two identifications in top degree yields the involution $\delta: \St(n)\to \St(n)$ sending a basis to its dual basis, also used extensively in \cite{CR}.

    \section{Trace-fixed parts of Milnor $K$-theory} \label{appendix:a1}

    In \cite[\S4]{CR}, Charlton--Rudenko construct an injective map from $\St_n(\Q)$ to a quotient of the Milnor $K$-theory of $\Q(z_1,\ldots, z_n)$ (i.e. the function field of $\G_m^n$); as noted there, this map is closely related to a map we constructed in \cite{X2}. In this section, we observe that our previous construction actually \emph{identifies} a module of spherical chains (a sort of ``oriented cover'' of the Steinberg module) with the \emph{trace-fixed} subspace of this Milnor $K$-group - and, in fact, the more refined motivic cohomology groups coming from toric hyperplane complements. These maps assemble into maps from a spherical chain complex (which covers the Orlik--Solomon complex) into a Gersten--like complex for $\G_m^n$; this may be viewed as a motivic refinement of the work on the ``Orlik--Solomon model'' for the (Hodge) cohomology of hypersurface complements of Dupont \cite[Chapter 3]{Dup} in the particular setting of toric hyperplane arrangements.

    Let $C(n)$ be the augmented simplicial homology complex of the ind-triangulation of $S^{n-1}$ at all points corresponding to rational rays of \cite{X2} (where it is called $\widetilde{\mathrm{Chains}}(n)$). We have the following summary of the results of \cite[\S2.2]{X2}:

    \begin{prop}
        The module $C(n)_k$ is presented by symbols $\Delta(r_1,\ldots, r_{k+1})$ for rational rays $r_1,\ldots, r_{k+1}$ lying in the same strict hemisphere, antisymmetric in these rays, subject to the stellar subdivision relations
        \begin{equation}
            \sum_{i=0}^{k+1} (-1)^i \Delta(r_0, \ldots, \hat{r}_i,\ldots, r_k) = 0
        \end{equation}
        for any $r_0,\ldots, r_{k+1}$ lying in the same strict hemisphere of a $(k-1)$-dimensional great circle; here, we take any terms with linearly dependent rays to be zero.
    \end{prop}

    \begin{rem}
    In fact, though we did not realize this at the time of writing \cite{X2}, only the \emph{rank $2$ stellar subdivision relations} (where $r_0$ lies on the great circle through $r_1$ and $r_2$) are necessary, as any simplicial subdivision can be barycentrically completed, resulting in its decomposition into simplicial subdivisions for which the subdividing vertex lies on an edge of a simplex. This is the analogue of the Bykovskii-type/quadratic generation results for spherical chains.
    \end{rem}

    As we observed previously in \cite[\S4]{X2}, we have a $\GL_n(\Q)$-equivariant isomorphism
    \[
     \rm{gr}^n\,\cal{C}(V) \xrightarrow{\sim} C(n)_{n-1},\; [C(v_1,\ldots, v_n)] \mapsto \mathrm{sgn}\,\det(v_1,\ldots, v_n) [S^{n-1} \cap C(v_1,\ldots, v_n)]
    \]
    for linearly independent $v_1,\ldots, v_n$.\footnote{There is no natural way to extend this to an isomorphism of graded modules, since the natural $\GL_n(\Q)$-action on functions of closed cones does not see the sign coming from orientation reversal: one could make a lot of choices of orientations and do so, and thereby get a cdga structure on both sides by transport of structure applied to the differential (from the chains side) and the product (from the cone algebra side), but it would be badly non-equivariant.}
    
    There is an equivariant surjection of cdgas $C(n)[1]\to \OS_n(\Q)$, which we also defined in loc. cit., given by sending $\Delta(r_1,\ldots, r_k)\mapsto [\Q r_1,\ldots, \Q r_k]$ where here our notation indicates the rational line parallel to the respective ray, killing the unique nonzero homology class in degree $n$ of $C(n)[1]$ coming from the fundamental class of $S^{n-1}$; under $\rm{gr}\,\cal{C}(V) \xrightarrow{\sim} C(n)[1],$, this is identified with the isomorphism $\rm{gr}\,\cal{C}(V)/F^D_{n-1} \xrightarrow{\sim} \OS_n(\Q)$ of the previous section.
    
    The scheme $T:=\G_m^n=\mathrm{Spec}\,\Z[z_1^{\pm},\ldots, z_n^{\pm}]$ has Gersten complex in Milnor $K$-theory
    \[
    \mathcal{K}(T):=(K_n \xrightarrow{\partial} K_{n-1} \xrightarrow{\partial} \ldots \xrightarrow{\partial} K_0)
    \]
    where $K_i := \bigoplus_{x\in T^{[n-i]}} K_i^M(k(x))$ with $T^{[n-i]}$ is the set of codimension-$(n-i)$ points of $T$, and $k(x)$ the residue field of $x$; the connecting arrows are residue/tame symbol maps. For each matrix $M\in M_{k\times n}(\Z)$, we have a finite isogeny $M:\G_m^k\to T$ given by pushforward; this is a finite map onto a dimension-$k$ (relatively, over $\Z$) sub-algebraic torus when $M$ is rank $k$ over $\Q$. We write a superscript $(0)$ on (co)homology groups attached to $\G_m^n$ to indicate the submodule on which scalar matrices $[a],a\in \mathbb{N}\setminus \{0\}$ act by $1$; such groups attain a pushforward action of $\GL_n(\Q)$ and not just $\GL_n(\Z)$.

    It is then shown in \cite[Theorem 3.1]{X2} that there is a $\GL_n(\Q)$-equivariant map of chain complexes
    \[
        \widetilde{\Psi}: C(n)[1]\to \mathcal{K}(T)^{(0)}, \;\Delta(r_1,\ldots, r_k) \mapsto \begin{pmatrix} r_1 & \ldots & r_k\end{pmatrix}_* \{ 1-z_1,\ldots, 1-z_k\}.
    \]

    \begin{rem}
    Note that in top degree $\widetilde{\Psi}$ \emph{does not} precisely lift the map of \cite[Theorem 46]{CR} from the Steinberg symbol, which sends a unimodular symbol to
    \[
    [v_1,\ldots, v_n] \mapsto \{1-\underline{z}^{v_1}\,\ldots, 1-\underline{z}^{v_n}\}
    \]
    (after quotienting by relations in the target). Our map is instead the composite of theirs by the involution $\delta: \St(n)\to \St(n)$.
    \end{rem}
    
    We consider the subcomplex $K(T_\Q)$ of $K(T)$ associated to the rationalization of the torus $T$; i.e., omitting points corresponding to subvarieties which are not flat over $\Z$. The image of $\widetilde{\Psi}$ visibly lands in $K(T_\Q)$, and we have the following refinement of the injectivity statement of Charlton--Rudenko:

    \begin{thm} \label{thm:main}
        The map $\widetilde{\Psi}$ yields an isomorphism between $C(n)[1]$ and $K(T_\Q)^{(0)}$. 
    \end{thm}
    \begin{proof}
    We proceed by induction: when $n=0$, this is the trivial assertion that $\Z$ is isomorphic to $\Z$. Assume now that we know the assertion for $n-1$: by the lemmas below, we have a commutative diagram 
    \begin{equation}
        \begin{tikzcd}
            C(n)_{n-1} \arrow[r,"\partial"] \arrow[d,"\widetilde{\Psi}"] & C(n)_{n-2} \arrow[d]\arrow[d,"\widetilde{\Psi}"] \\
            K_n^M(\G_m^n)^{(0)} \arrow[r,"\partial"] & \bigoplus_{A\in } K_{n-1}^M(\G_m^{n-1})^{(0)}
        \end{tikzcd}
    \end{equation}
    where the kernel of the top row is the fundamental class $[S^{n-1}]$, and the kernel of the bottom row is $\Z \cdot \{-z_1,\ldots, -z_n\}$. By the inductive hypothesis, the right vertical arrow is an isomorphism, so to show the same for the left, it suffices to see that $[S^{n-1}]$ maps to $\{-z_1,\ldots, -z_n\}$, which we showed in \cite[\S3.1]{X2}.
    \end{proof}

    In light of \eqref{eq:rel}, our theorem is a direct refinement of \cite[Theorem 46, Proposition 47]{CR}. We conclude by proving the lemmas used in the above proof:

    \begin{lem}
        The module $K_i^{(0)}$ in $K(T_\Q)^{(0)}$ is supported only on summands coming from the generic points of \emph{subgroup tori}: that is, ``linear'' embeddings $(\G_m^i)_\Q \to T_\Q$ associated to $i$ by $n$ matrices. In particular, we may identify the term
        \[
            \left(\bigoplus_{x\in T_\Q^{[n-i]}} K_i^M(k(x))\right)^{(0)}= \bigoplus_{A\in } K_i^M(\G_m^{i})^{(0)}
        \]
        in $K(T_\Q)^{(0)}$. 
    \end{lem}
    \begin{proof}    
        Suppose a trace-fixed element $s$ is supported on a union $C$ of dimension-$i$ closed cycles. Then $[a]_*s=s$ is necessarily supported on a subscheme of $[a]_*C$, so we have $C\subset [a]_*C$ for any $a\in \mathbb{N}$. Since $C$ is a union of finitely many irreducible components and each $[a]_*$ is a finite map, we must actually have $C = [a]_*C$ for each $a$, with each of these isogenies permuting the irreducible components. 
        
        We claim the only such cycles are unions of algebraic tori subgroups, as in the lemma statement: in fact, we claim this is even true for the complexification $T_\C$. Let a \emph{torsion subtorus} of $\G_m^n$ to be a subvariety of the form $\zeta + H$ for a subgroup torus $H\subset \G_m^n$ and a torsion point $\zeta$. 
        
        We claim that it suffices to show $C$ is a union of torsion subtori of dimension $i$. Indeed, suppose that we have some torsion subtorus $\zeta+H$ as a component of $C$, for $\zeta$ of minimal order $c$. Then $C$ must contain a component of $[a]^{-1}(\zeta+H)$ for every $a\in \mathbf{N}$, in particular for $a=c^k$ for all $k\ge 0$. These components are precisely $\zeta'+H$ as $\zeta'$ ranges over $c^k$-th roots of $\zeta$; if $c>1$, these cosets are all distinct as $\zeta$ and $k$ vary, because the $[c^k]$-preimages of the class of the $c$-torsion point $\zeta$ in $\G_m^n/H$ are all distinct - this forces $C$ to have infinitely many components, a contradiction. We conclude that $c=1$, and thus $C$ is a union of subgroup tori of dimension $i$.

        We now prove that the components of any $i$-dimensional $C$ are torsion subtori of dimension $i$ for $0\le i \le n$, by induction on $n$; the case $n=0$ is clear. Assuming the inductive hypothesis for $n-1$, we consider a trace-fixed $C$ of dimension $1\le i \le n$; if $i=n$, the result is obvious. Otherwise, if each component of $C$ is contained in a torsion subtorus of dimension $n-1$, the result follows by the inductive hypothesis. 
        
        Finally, if there is a component of $C$ not contained in any torsion subtorus of dimension $n-1$, this means this component contains a point $x=(x_1,\ldots, x_n)$ such that there is no non-zero vector $(v_1,\ldots, v_n)\in \Z^n$ for which $x_1^{v_1}\cdot \ldots \cdot x_n^{v_n}=1$. But then the powers $x^k$ of this point are all also contained in $C$, and these are Zariski dense in $\G_m$: to see this, suppose that
        \[
            f(z_1,\ldots, z_n) = \sum_{\underline{i}\in S\subset \Z_{\ge 0}^n} a_{\underline{i}} z^{\underline{i}}
        \]
        vanishes on all $x^k$. Note that no two terms of the form $x^{k\underline{i}}$ are equal by hypothesis; thus, the $|S|$ by $|S|$ Vandermonde matrix of the first $|S|$ powers of the set $(x^{\underline{i}})_{\underline{i}\in S}$ is invertible, and the linear relations $(f(x^{k \underline{i}}))_{\underline{i}\in S}$ for $k=1,2,\ldots, |S|$ imply that every coefficient $a_{\underline{i}}=0$, meaning that only the zero polynomial vanishes on all powers of $x$. Thus, this case cannot occur.
    \end{proof}
    
    \begin{lem}
    The kernel of 
    \[
        K_n^{(0)} \to K_{n-1}^{(0)}
    \]
    in $K(T_\Q)$ is the free rank-$1$ submodule spanned by $\{-z_1,\ldots, -z_n\}$.
    \end{lem}
    \begin{proof}
        We proceed by induction on $n$. When $n=1$, the kernel of 
        \[
        K_1^M(\Q(z))\xrightarrow{\partial} \bigoplus_{x\in \G_m^{[1]}} K_0^M(k(x))
        \]
        is the rational functions with no poles at any points of $\G_m$; certainly these are only of the form $C\cdot z^k$ for some constant $C$ and integer exponent $k$. One can compute that the trace-fixed ones are precisely of the form $(-z)^\Z$.
        
        A result of Milnor (see, e.g., \cite[Theorem 7.2.1]{GS}) tells us that for a field $F$, we have an exact sequence
        \[
        0\to K_n^M(F)\xrightarrow{\iota_*} K_n^M(F(t)) \xrightarrow{\partial} \bigoplus_{f\text{ monic irred}} K_{n-1}^M(F[t]/(f))\to 0
        \]
        split by a ``specialization at infinity'' map $K_n^M(F(t))\to K_n^M(F)$ sending $-t$ to zero when $n=1$. Take now $F=\Q(z_1,\ldots, z_{n-1}$, and relabel the parameter $t$ as $z_n$, so that the middle term becomes $K_n^{(0)}$ as in the lemma statement. Observe that the image of $\iota_*$ is transverse to $K_n^{(0)}$, since the factor of $[a]$ given by the diagonal matrix $\mathrm{diag}(1,1,\ldots, ,1,a)_*$ acts on it by the scalar $a$, and that any element in $\ker \left(K_n^{(0)}\to K_{n-1}^{(0)}\right)$ definitionally has trivial image in every summand on the right except possibly $f=(z_n)$, the sole codimension-$1$ point which does not restrict to a codimension-$1$ point of $\G_m^n$. Moreover, its image in $K_{n-1}^M(\G_m^{n-1})^{(0)} \cong K_{n-1}^M(F[z_n]/(z_n))^{(0)}$ likewise has trivial tame symbol at every codimension $1$ point, since every codimension-$1$ subvariety of $\G_m^{n-1}$ is the intersection of the hyperplane with a codimension-$1$ subvariety of $\G_m^n$ (and the commutativity of the corresponding tame symbol maps). We therefore obtain an isomorphism 
        \[
            \ker \left(K_n^{(0)}\to K_{n-1}^{(0)}\right)\xrightarrow{\partial_{\{z_n=0\}}} \ker\left(K_{n-1}^M(\Q(z_1,\ldots, z_{n-1}))^{(0)}\xrightarrow{\partial} \bigoplus_{x\in (\G_m^{n-1})^{[1]}} K_{n-2}^M(k(x))\right) 
        \]
        whose inverse is cup product with $-z_n$, and the result follows by the inductive hypothesis.
        
    \end{proof}

    Note that in \cite{X2}, we actually consider realizations of spherical chains/Steinberg-type symbols in trace-fixed parts of colimits of motivic cohomology groups of complements of toric hypersurface arrangements, instead of passing to generic points and taking Milnor $K$-theory. On trace-fixed parts, these should also be isomorphic to the spaces of symbols, by a somewhat more careful induction argument using coniveau spectral sequences as in \cite{SV}, though we do not write this here. We expect that this even holds for finite hyperplane complements and finite submodules of $C(n)$, by the same methods. This observation came from discussions with Romyar Sharifi.

    \printbibliography
\end{document}